\documentclass[11pt]{article}

\usepackage{geometry}

\usepackage{amssymb}
\usepackage{latexsym}
\usepackage{graphics}
\usepackage{amsmath,amsfonts,amsthm}
\usepackage{comment}

\begin{document}

\setlength{\baselineskip}{1.3\baselineskip}
\setlength{\parskip}{0.2\baselineskip}

\newtheorem{theorem}{Theorem}
\newtheorem{corollary}[theorem]{Corollary}
\newtheorem{lemma}[theorem]{Lemma}
\newtheorem*{numlesstheorem}{Theorem $\textrm{2}'$}
\renewcommand{\thefootnote}{\alph{footnote}}

\newcommand{\beqaa}{\begin{eqnarray}}
\newcommand{\beqae}{\begin{eqnarray*}}
\newcommand{\eeqae}{\end{eqnarray*}}
\newcommand{\eeqaa}{\end{eqnarray}}

\newcommand{\con}{\rightarrow}
\newcommand{\pbrack}[1]{\left( {#1} \right)}
\newcommand{\sbrack}[1]{\left[ {#1} \right]}
\newcommand{\key}[1]{\left\{ {#1} \right\}}
\newcommand{\remark}{{\bf Remark}\quad}
\newcommand{\example}{\noindent {\bf Example:\ }}
\newcommand{\ds}{\displaystyle}

\newcommand{\R}{{\mathbb R}}
\newcommand{\C}{{\mathbb C}}
\newcommand{\N}{{\mathbb N}}
\newcommand{\Z}{{\mathbb Z}}


\title{On Stability of Volterra Difference Equations of Convolution 
Type}
\author{
\small \bf Higidio Portillo Oquendo,  Jos\'e R. Ramos Barbosa\\
\small Federal University of Paran\'a, Brazil \\
\small and \\
\small \bf Patricia S\'anez Pacheco\\
\small Federal Technological University of Paran\'a, Brazil
}
\date{}
\maketitle

\let\thefootnote\relax\footnote{{\it Email address:} {\rm higidio@ufpr.br} (Higidio Portillo Oquendo), {\rm jrrb@ufpr.br} (Jos\'e R. Ramos Barbosa), {\rm patricias@utfpr.edu.br} (Patricia S\'anez Pacheco)}

\begin{abstract}
In \cite{Elaydi-10}, S.\ Elaydi obtained a characterization of the stability of
the null solution of the Volterra difference equation
\beqae
x_n=\sum_{i=0}^{n-1} a_{n-i} x_i\textrm{,}\quad n\geq 1\textrm{,}
\eeqae
by localizing the roots of its characteristic equation
\beqae
1-\sum_{n=1}^{\infty}a_nz^n=0\textrm{.}
\eeqae
The assumption that $(a_n)\in\ell^1$ was the single hypothesis considered for 
the validity of that characterization, which is an insufficient condition if the
ratio $R$ of convergence of the power series of the previous equation equals 
one. In fact, when $R=1$, this characterization conflicts with a result obtained
by Erd\"os et al in \cite{Erdos}. Here, we analyze the $R=1$ case and show that
some parts of that characterization still hold. Furthermore, studies on 
stability for the $R<1$ case are presented. Finally, we state some new results 
related to stability via finite approximation.
\end{abstract}



\section{Introduction}
\noindent In the present work, we analyze the stability of the null solution of Volterra difference equations of convolution type,
\beqaa\label{Eq-Recu-10}
x_n=\sum_{i=0}^{n-1}a_{n-i}x_i,\quad n\geq1\textrm{,}
\eeqaa
whose recursive process starts at $x_0\in\R$. Several results related to 
this subject matter circulates in the specialized scientific literature. One of
most well-known is the following theorem:
\begin{theorem}[See \cite{Kolmanovskii}]\label{Teo-Abs-10}
If $\ds\sum_{n=1}^{\infty}|a_n|<1$, then the null solution of (\ref{Eq-Recu-10})
is asymptotically stable.
\end{theorem}
We now present another important characterization of the stability of the null
solution of (\ref{Eq-Recu-10}) as obtained by S.\ Elaydi in \cite{Elaydi-10}.
Let $(x_n)$ be a solution of (\ref{Eq-Recu-10}) with initial condition $x_0=1$.
Consider the two power series
\beqae
x(z):=\sum_{n=0}^{\infty}x_nz^n,\quad a(z):=\sum_{n=1}^{\infty}a_nz^n\textrm{.}
\eeqae
Then, formally, such series satisfies
\beqaa\label{Eq-Ser-Pot-10}
x(z)(1-a(z))=1\textrm{.}
\eeqaa
Thus the coefficients of $x(z)$ can be found by determining the coefficients of
the power series representation of the function $(1-a(z))^{-1}$. Hence the roots
of the characteristic equation
\beqaa\label{Caract-10}
1-\sum_{n=1}^{\infty}a_nz^n=0
\eeqaa
play an important role in this sense. By this reasoning, S.\ Elaydi obtained 
necessary and sufficient conditions for the stability of the null solution
by localizing the roots of $1-a(1/z)=0$ with respect to the set
$\{z\in\C\,:\,\|z\|\geq1\}$. We now enunciate the result obtained by Elaydi with
a variable change by writing $z$ in place of $1/z$ (as considered in 
\cite{Elaydi-10}). In the following, we use the notation:
\beqae
B_r(z_0)=\{z\in\C : |z-z_0|<r\},\quad r>0\textrm{.}
\eeqae

\begin{theorem}[See \cite{Elaydi-10,Elaydi-2009}]\label{Teo-Elaydi-10}
Let $(a_n)\in\ell^1$. Then:
\begin{itemize}
\item[{\normalfont (a)}] The null solution of (\ref{Eq-Recu-10}) is stable if, and only if, the
characteristic equation (\ref{Caract-10}) has no roots in $B_1(0)$ and its
possible roots in $|z|=1$ are of order $1$.
\item[{\normalfont (b)}] The null solution of (\ref{Eq-Recu-10}) is asymptotically stable if, and
only if, the characteristic equation (\ref{Caract-10}) has no roots in
$\overline{B_1(0)}$.
\end{itemize}
\end{theorem}

\noindent It is worth mentioning that the previous theorem has also appeared as theorems 6.16 and 6.17 in \cite{Elaydi-2005}. Furthermore, as a consequence of 
theorem \ref{Teo-Elaydi-10}, Elaydi set the following result on asymptotic 
instability:

\begin{theorem}[See \cite{Elaydi-10,Elaydi-2009}]\label{Teo-Elaydi-20}
If $(a_n)\in\ell^1$ is a sequence whose terms do not change signs for
$n\geq1$, then the null solution of (\ref{Eq-Recu-10}) is not asymptotically
stable if one of the following conditions is satisfied:
\begin{itemize}
\item[{\normalfont (a)}] $\sum_{n=1}^{\infty}a_n\geq1$;
\item[{\normalfont (b)}] $\sum_{n=1}^{\infty}a_n\leq-1$ and $a_n>0$ for some $n\geq 1$;
\item[{\normalfont (c)}] $\sum_{n=1}^{\infty}a_n\leq -1$ and $a_n<0$ for some $n\geq 1$ and
$\sum_{n=1}^{\infty}a_n$ is sufficiently small.
\end{itemize}
\end{theorem}

\noindent At this point, if we consider the sequence
\beqaa\label{Seq-Conf-10}
a_n=\frac{1}{n(n+1)},\quad n\geq 1\textrm{,}
\eeqaa
the null solution of (\ref{Eq-Recu-10}) is not asymptotically stable by the item
(b) of theorem \ref{Teo-Elaydi-10} or the item (a) of theorem 
\ref{Teo-Elaydi-20}. On the other hand, (\ref{Seq-Conf-10}) satisfies the 
conditions for asymptotic stability of the null solution of (\ref{Eq-Recu-10})
as given by the following theorem due to Erd\"os, Feller e Pollard:
\begin{theorem}[See \cite{Erdos}]\label{Teo-Erdos}
Let $(a_n)$ be a sequence of nonnegative terms such that
\beqae
\mathrm{gcd}\{n\in\N\,:\,a_n>0\}=1\textrm{,}\quad
\sum_{n=1}^{\infty}a_n=1\quad
\textrm{and}\quad
\sum_{n=1}^{\infty}na_n=\infty\textrm{.}
\eeqae
Then the null solution of (\ref{Eq-Recu-10}) is asymptotically stable.
\end{theorem}
\noindent Therefore there exists a contradiction between the previous theorem and 
theorems \ref{Teo-Elaydi-10} and \ref{Teo-Elaydi-20}. An analysis of the proof 
of theorem \ref{Teo-Elaydi-10} makes clear that the analyticity of the power 
series $a(z)$ on the circumference $|z|=1$ was strongly used. But this fact is 
not a consequence of the assumption that $\left(a_n\right)\in\ell^1$, as we can
see in the example (\ref{Seq-Conf-10}). Hence a simple correction can be made by
introducing the radius of convergence of the series $a(z)$,
\beqaa\label{Raio-10}
\frac{1}{R}=\limsup_{n\con\infty}\sqrt[n]{|a_n|}\textrm{,}
\eeqaa
and replacing the hypothesis that $\left(a_n\right)\in\ell^1$ by $R>1$. In fact,
if $R>1$, then the function $a(z)$ is analytic in $|z|=1$ and $(a_n)\in\ell^1$,
which are conditions that assure us of the validity of theorem 
\ref{Teo-Elaydi-10}. Furthermore, by applying this new hypothesis, there is no 
contradiction between theorem \ref{Teo-Erdos} and theorems \ref{Teo-Elaydi-10} 
and \ref{Teo-Elaydi-20} since we may easily show that the conditions of theorem
\ref{Teo-Erdos} implies that $R=1$.\\ 
Therefore, if $R=1$, the validity of theorem \ref{Teo-Elaydi-10} is an open 
problem since the analyticity on the unit circumference can not be applied. 
Besides theorem \ref{Teo-Erdos}, some other results give us some sufficient 
conditions for the asymptotic stability of the null solution of 
(\ref{Eq-Recu-10}) and can be found in
\cite{Berenhaut,Elaydi-MV,Vecchio,Kolmanovskii-K,Choi,Tang,Nguyen,Nigmatulin}.\\
The rest of this paper is divided as follows: In section 2 we present an 
alternative proof of theorem \ref{Teo-Elaydi-10} with the corrected hypothesis,
that is, we suppose that $R>1$. In section 3 we analyze the validity of theorem
\ref{Teo-Elaydi-10} when we have $R= 1$. Furthermore, we study possible 
characterizations for the null solution of (\ref{Eq-Recu-10}) to be stable if 
$R<1$. In section 4 we analyze the stability via finite approximations.
\section{An Alternative Proof of Theorem \ref{Teo-Elaydi-10}}
\noindent Firstly, note that the stability of the null solution of 
(\ref{Eq-Recu-10}) depends on the behavior of the particular solution $(x_n)$ 
with initial condition $x_0=1$. In fact, an arbitrary solution to the equation 
(\ref{Eq-Recu-10}) with initial condition $\beta$ is given by $(x_n\beta)$. 
Thus, in what follows, $(x_n)$ denotes the solution of (\ref{Eq-Recu-10}) with 
initial condition $x_0=1$. Therefore, just for future reference, we have the
following elementary result, which is already known in a more general case (see e.g.\ theorem 3.1 in \cite{Crisci-Vecchio}):
\begin{theorem}\label{Teo-Higidio}
The null solution of (\ref{Eq-Recu-10}) is:
\begin{enumerate}
\item stable if, and only if, $(x_n)$ is bounded;
\item asymptotically stable if, and only if, $x_n\con 0$.
\end{enumerate}
\end{theorem}
\noindent Now, theorem \ref{Teo-Elaydi-10} was proved by Elaydi in 
\cite{Elaydi-10} using strongly $Z$-transform techniques and the analyticity of
the function $a(z)$ on the unit circumference. Therefore, if $R>1$, the 
arguments presented there, along with the replacement of $1/z$ by $z$, give us 
a proof of this theorem. However, due to the importance of this result and for 
the sake of completeness, we will prove theorem \ref{Teo-Elaydi-10} using 
alternative arguments which are different from the ones used by Elaydi.
\begin{numlesstheorem}
Replace $\left(a_n\right)\in\ell^1$ by $R>1$. Then the items (a) and (b) of 
theorem \ref{Teo-Elaydi-10} are valid.
\end{numlesstheorem}
\begin{proof}
We show both items, (a) and (b), simultaneously.\\
{\bf Sufficient Condition:} Suppose that the equation $1-a(z)=0$ has no roots in
$B_1(0)$. Consider $\rho\in]1,R[$ such that the function $1-a(z)$ is not zero in
$1<|z|\leq\rho$, denoting by $z_1,\ldots,z_s$ its possible zeros in $|z|=1$. 
For each $n\in\N$, define the function $h_n(z)=x(z)/z^{n+1}$. Hence
\beqae
\underset{z=0}{\text{Res}}\;h_n(z)=\frac{x^{(n)}(0)}{n!}=x_n\textrm{.}
\eeqae
By this identity and the Residue Theorem, it follows that
\beqaa\label{Eq-Res-rn-10}
x_n
=
\frac{1}{2\pi i}\int_{|z|=\rho} h_n(z)\;dz
-
\sum_{k=1}^s\underset{z=z_k}{\text{Res}}\;h_n(z)
\textrm{,}
\eeqaa
whose integral can be estimated by
\beqaa\label{Eq-Res-rn-30}
\left|\int_{|z|=\rho}h_n(z)\;dz\right|
=
\left|
\int_{0}^{2\pi}\frac{ix(\rho e^{i\theta})}{\rho^ne^{in\theta}}\;d\theta
\right|
\leq 
\frac{2\pi}{\rho^n}\sup_{|z|=\rho}|x(z)|\con 0
\eeqaa
when $n\con\infty$. Then, if $1-a(z)$ is not zero in $|z|=1$, we have $x_n\con0$
as $n\con\infty$. Therefore the null solution of (\ref{Eq-Recu-10}) is 
asymptotically stable. On the other hand, if the zeros of $1-a(z)$ in $|z|=1$ 
are of order one, we have that, for each $1\leq k\leq s$, $x(z)=q_k(z)/(z-z_k)$
where $q_k(z)$ is an analytic function at the point $z_k$ with 
$q_k(z_k)\ne0$. Hence
\beqae
\underset{z=z_k}{\text{Res}}\;h_n(z)
=
\frac{q_k(z_k)}{z_k^{n+1}}
\quad\Rightarrow\quad
\left|\underset{z=z_k}{\text{Res}}\;h_n(z)\right|=|q_k(z_k)|
\textrm{.}
\eeqae
From this equation and formulas (\ref{Eq-Res-rn-10}) and (\ref{Eq-Res-rn-30}), 
it follows that $(x_n)$ is bounded. Therefore the null solution of 
(\ref{Eq-Recu-10}) is stable.\\
{\bf Necessary Condition:} Suppose now that the null solution of 
(\ref{Eq-Recu-10}) is stable. In theorem \ref{Teo-NoEstavel-10} we prove that 
the function $1-a(z)$ is not zero in $B_1(0)$. So it remains to show that the 
possible zeros of $1-a(z)$ in $|z|=1$ are of order $1$. In fact, if some $z_k$
is a zero of $1-a(z)$ of order $1+m_k\geq2$, then $x(z)=q_k(z)/(z-z_k)^{1+m_k}$
where $q_k(z)$ is analytic at $z_k$ with $q_k(z_k)\ne0$. Hence
\beqae
\underset{z=z_k}{\text{Res}}\;h_n(z)
=
\frac{1}{m_k!}\frac{d^{m_k}}{dz^{m_k}}\sbrack{\frac{q_k(z)}{z^{n+1}}}_{z=z_k}
=
n^{m_k}\alpha_k e^{-i\theta_k n}+p_k(n)
\textrm{,}
\eeqae
where $z_k=e^{i\theta_k}$ with $\theta_k\in [0,2\pi[$, $\alpha_k\ne 0$ and $p_k$
is a polynomial of degree less than $m_k$. Reordering the zeros of $1-a(z)$ 
such that $z_1,\ldots,z_\mu$, $\mu\leq s$, are all the highest order zeros,
$1+m$, we have that
\beqae
\ds\sum_{k=1}^{\mu}\underset{z=z_k}{\text{Res}}\;h_n(z)
=
n^m\left(\sum_{k=1}^{\mu}\alpha_ke^{-i\theta_kn}\right)
+
p(n)
\textrm{,}
\eeqae
where $p$ is a polynomial of degree less than $m$. Since
\beqae
\sum_{k=1}^{\mu}\alpha_ke^{-i\theta_kn}\neq0
\quad
\textrm{for infinitely many indices}
\quad
n
\textrm{,}
\eeqae
it follows that $\ds\sum_{k=1}^{\mu}\underset{z=z_k}{\text{Res}}\;h_n(z)$ is 
unbounded. Then, since 
$\ds\sum_{k=\mu+1}^{s}\underset{z=z_k}{\text{Res}}\;h_n(z)$ has powers of $n$ 
less than $m$, it follows from (\ref{Eq-Res-rn-10}) and (\ref{Eq-Res-rn-30}) 
that $(x_n)$ is unbounded, which is absurd. If we now consider that the null 
solution of (\ref{Eq-Recu-10}) is asymptotically stable, then it is stable and 
therefore $1-a(z)$ has no zeros of order greater than one in $|z|=1$. Thus, 
since
\beqae
\sum_{k=1}^{s}\underset{z=z_k}{\text{Res}}\;h_n(z)
=
\sum_{k=1}^{s}\alpha_k e^{-i\theta_k n}\not\con 0
\textrm{,}
\eeqae
from (\ref{Eq-Res-rn-10}) and (\ref{Eq-Res-rn-30}) we have that $1-a(z)$ cannot
have zeros in $|z|=1$.
\end{proof}
\section{Results on Stability when $R=1$ or $R<1$}
\noindent In this section, we analyze the validity of theorem 
\ref{Teo-Elaydi-10} in the case where $R=1$ and provide some results on 
stability/instability of the null solution of (\ref{Eq-Recu-10}) when $R<1$.
It is worth pointing out that the arguments to be used are still valid in the 
case where $R>1$. Initially we give two examples on stability/instability 
for the case with $R<1$.\\
\example Let $p>1$. At first consider the sequence $a_n=-p^n$. Then $R<1$ and 
the solution of (\ref{Eq-Recu-10}) with initial condition $x_0=1$ is
$(x_n)=(1,-p,0,0,\ldots)$, which converges to zero. So the null solution is 
asymptotically stable. On the other hand, if we consider
\beqaa\label{Exem-10}
a_n=\frac{p^{n-1}}{n(n+1)}\textrm{,}\quad n\geq1\textrm{,}
\eeqaa
we also have $R<1$. However, since each term in this sequence is positive, 
it follows (by induction) that each term of $(x_n)$ is positive. Hence 
$x_n\geq a_n$ for each $n\geq 1$. Thus $(x_n)$ is unbounded and consequently 
the null solution of (\ref{Eq-Recu-10}) is unstable by theorem 
\ref{Teo-Higidio}.\\ 
The previous example shows that the stability or instability of the null 
solution does not depend on the radius of convergence of $a(z)$. In what follows
we present a result on instability which does not depend on the size of $R$.
(In particular, it holds if $R=1$ or $R<1$, which is the case where theorem $\textrm{2}'$ is not applicable.) 
\begin{theorem}\label{Teo-NoEstavel-10}
Let
$\ds D_a=\left\{z\in\C\,:\,\sum_{n=1}^{\infty}a_nz^n\text{ converges}\right\}$. 
If the characteristic equation (\ref{Caract-10}) has a root in $D_a\cap B_1(0)$,
then the null solution of (\ref{Eq-Recu-10}) is unstable.
\end{theorem}
\begin{proof}
Denote by $\rho_0 e^{i\theta}$, $\rho_0\in]0,1[$, one of the zeros of $1-a(z)$ 
of smallest modulus. Hence $1-a(z)\ne 0$ for each $z\in B_{\rho_0}(0)$. Since 
$x(z)=(1-a(z))^{-1}$ for every $z\in B_{\rho_0}(0)$, we have that
\beqae
\lim_{\rho\con \rho_0^-}|x(\rho e^{i\theta})|=\infty\textrm{.}
\eeqae
Therefore the radius of convergence of $x(z)$ is not greater than $\rho_0$. Then
\beqae
\limsup\sqrt[n]{|x_n|}\geq\frac{1}{\rho_0}>1\textrm{.}
\eeqae
Thus, if $\frac{1}{\rho_0}>\beta>1$, there exists a subsequence $(x_{n_k})$ for 
which $\sqrt[n_k]{|x_{n_k}|}>\beta$. So we conclude that 
$|x_{n_k}|>\beta^{n_k}$. Therefore the sequence $(x_n)$ is not bounded.
\end{proof}
\noindent{\bf Remark:} The converse of theorem \ref{Teo-NoEstavel-10} is not 
valid. In fact, for the sequence given in (\ref{Exem-10}), the null solution of
(\ref{Eq-Recu-10}) is unstable. On the other hand, since 
$D_a=\overline{B_{1/p}(0)}$ and
\beqae
|a(z)|=\left|\sum_{n=1}^{\infty}\frac{p^{n-1}}{n(n+1)}z^n\right|\leq 
\frac{1}{p}<1\textrm{,}\quad\forall z\in \overline{B_{1/p}(0)}\textrm{,}
\eeqae
it follows that $1-a(z)$ is not zero in $D_a\cup B_1(0)$.
\begin{corollary}\label{MaiorUm-10}
If $\ds\sum_{n=1}^{\infty}a_n>1$ converges, then the null solution of 
(\ref{Eq-Recu-10}) is unstable.
\end{corollary}
\begin{proof}
Since $\ds\sum_{n=1}^{\infty}a_n$ converges, it follows from Abel's theorem that
the power series $a(z)$ is continuous in $[0,1]$. Specifically, one has that
$\lim\limits_{z\to 1^{-}}a(z)=\sum\limits_{n=0}^{\infty}a_n$. So consider 
the function $b(z)=1-a(z)$. Then $b(0)=1$ and $b(1)<0$. Therefore $b(z)$ has a
zero in the interval $]0,1[\subset B_1(0)$. It follows from the previous theorem
that the null solution of (\ref{Eq-Recu-10}) is not stable.
\end{proof}
\noindent{\bf Remark:} By item (a) of theorem \ref{Teo-Elaydi-20}, we have that 
the lack of asymptotic stability takes place when, in particular, the hypothesis
of the previous corollary holds, provided that the terms of the sequence $(a_n)$
do not change signs. So, the previous corollary states the lack of stability 
(and therefore the lack of asymptotic stability) without any sign-preserving 
condition. Furthermore, that corollary remains valid if we replace the 
hypothesis $\sum\limits_{n=1}^{\infty}a_n>1$ by
$\sum\limits_{n=1}^{\infty}(-1)^na_n<1$.\\ \\
The following theorem shows that part of what was stated in the item (b) of
theorem $\textrm{2}'$ remains valid if $R=1$. (Once more, as seen in example 
\ref{Seq-Conf-10}, we emphasize that the item (b) of theorem \ref{Teo-Elaydi-10}
is not valid in the case where $R=1$. Furthermore, the argument which was used 
to prove that item in that case is not applicable since it was based on the 
analyticity of the function $a$ on the unit circumference. Anyway, that result
remains partially valid if $R=1$. We prove the sufficient condition of it  without using the analyticity argument.)
\begin{theorem}\label{Teo-Assint-10}
If $1-a(z)$ is continuous and not zero in $\overline{B_1(0)}$, then the null
solution of (\ref{Eq-Recu-10}) is asymptotically stable.
\end{theorem}
\begin{proof}
First we observe that the function $x(z)=(1-a(z))^{-1}$ is uniformly continuous
on $\overline{B_{1}(0)}$. Now, for each $\rho\in [0,1[$ and $n\in\N$, define the
following two functions on the interval $[0,2\pi]$:
\beqae
f_{\rho}(\theta)=x(\rho e^{i\theta})e^{-in\theta}\textrm{,}\quad 
f(\theta)=x(e^{i\theta})e^{-in\theta}\textrm{.}
\eeqae
It follows from the uniform continuity of $x(z)$ on $\overline{B_1(0)}$ that 
$f_{\rho}\con f$ uniformly on $[0,2\pi]$ as $\rho\con1^{-}$. Therefore
\beqae
\lim_{\rho\con 1^{-}}\int_0^{2\pi}x(\rho e^{i\theta})e^{-in\theta}\;d\theta
=
\int_0^{2\pi} x(e^{i\theta})e^{-in\theta}\;d\theta\textrm{.}
\eeqae
From the Cauchy's Integral Formula, we have that, for every $\rho\in]0,1[$,
\beqae
x_n
=
\frac{1}{2\pi i}\int_{|z|=\rho}\frac{x(z)}{z^{n+1}}\;dz
=
\frac{1}{2\pi \rho^n}\int_0^{2\pi} x(\rho e^{i\theta})e^{-in\theta}\;d\theta
\textrm{.}
\eeqae
Applying the limit when $\rho\con1^{-}$, it follows that
\beqae
x_n=\frac{1}{2\pi}\int_0^{2\pi}x(e^{i\theta})e^{-in\theta}\;d\theta\textrm{.}
\eeqae
So, by the Riemann-Lebesgue Lemma, one has $x_n\con 0$, which shows that the 
null solution of (\ref{Eq-Recu-10}) is asymptotically stable.
\end{proof}
\noindent{\bf Remark:} The conclusion of the preceding theorem is not valid if
the radius of convergence of $a(z)$, $R$, is less than one. In other words, 
even if $1-a(z)$ is continuous and not zero in $\overline{B_R(0)}$, the null
solution may not be asymptotically stable. To illustrate this statement, it 
suffices to consider the sequence given in (\ref{Exem-10}). Additionally, note 
that, if $R=1$, the converse of the preceding theorem is not valid, as shown in
example (\ref{Seq-Conf-10}).\\
To finalize this section, we enunciate an auxiliary lemma for the 
characterization of the stability (not necessarily an asymptotic one) of the 
null solution when $R=1$.
\begin{lemma}
Consider the following  power series
\beqae
y(z)=\sum_{n=0}^{\infty}y_n z^n,\quad p(z)=\sum_{n=0}^{\infty}p_n z^n\textrm{.}
\eeqae
Suppose that $(y_n)$ is bounded. Then:
\begin{enumerate}
\item If $p(z)=(1-e^{-i\theta}z)y(z)$, then $(p_n)$ is bounded.
\item If $p(z)=(1-z)y(z)$, then $p(z)$ is bounded on the interval $[0,1[$.
\end{enumerate}
\end{lemma}
\begin{proof}
First note that $p_n=y_n-e^{-i\theta}y_{n-1}$ for each $n\geq1$. Hence, 
if $C=\sup_{n\geq 0}|y_n|$, then $|p_n|\leq 2C$ for every $n\geq1$, which  
demonstrates the item 1. Consider now that the hypothesis of the item 2 is 
valid. It follows that $z\in [0,1[$ implies
\beqae
|p(z)|\leq (1-z)\sum_{n=0}^{\infty}|y_n|z^n\leq C(1-z)\sum_{n=0}^{\infty}z^n=C
\textrm{.}
\eeqae
\end{proof}
\noindent In what follows, suppose that the power series $a(z)$ converges on 
$\overline{B_1(0)}$ and the possible zeros of $1-a(z)$ occur at 
$e^{i\theta_1},\ldots,e^{i\theta_s}$, in other words,
\beqaa\label{Pos-zero-10}
1-a(z)=(1-e^{-i\theta_1}z)^{m_1}\cdots(1-e^{-i\theta_s}z)^{m_s}q(z)\textrm{,}
\eeqaa
where $q(z)$ is not zero in $\overline{B_1(0)}$. Furthermore, consider the space
\beqae
\mathbb{L}^1:=\key{q(z)=\sum_{n=0}^{\infty}q_nz^n\,:\,(q_n)\in\ell^1}\textrm{.}
\eeqae
Furthermore, again, we make clear that, the argument which was used to prove
the item (a) of theorem $\textrm{2}'$ is valid only if $R>1$ since it 
was based on the analyticity of the function $a$ on the unit circumference. 
Here, in the case where $R=1$, we demonstrate that result for a particular 
situation without using the analyticity argument.\\
In these conditions we have the following result: 
\begin{theorem}
Assume that the power series $a(z)$ converges on 
$\overline{B_1(0)}$ and $q\in \mathbb{L}^1$. The null solution of (\ref{Eq-Recu-10}) is stable 
if, and only if, the possible zeros of $1-a(z)$ as given in (\ref{Pos-zero-10}) 
are of order 1.
\end{theorem}
\begin{proof}
First consider that $m_1=\cdots=m_s=1$. Since $q\in\mathbb{L}^1$ and $q$ is not
zero in $\overline{B_1(0)}$, by Wiener's Theorem, we have that
\beqaa\label{Inv-20}
[q(z)]^{-1}=\hat{q}(z)=\sum_{n=0}^{\infty}\hat{q}_nz^n\in \mathbb{L}^1\textrm{.}
\eeqaa
On the other hand,
\beqaa\label{Inv-30}
\sbrack{\ds\prod_{j=1}^{s}(1-e^{-i\theta_j}z)}^{-1}
=
\sum_{n=0}^{\infty}\alpha_nz^n
\quad\text{with}\quad 
\alpha_n
=
\sum_{j=1}^s\pbrack{\frac{e^{i\pbrack{\sum_{\mu\ne j}\theta_{\mu}}}}{\prod_{\mu\ne j}(e^{i\theta_{\mu}}-e^{i\theta_j})}}e^{-i\theta_j n}
\textrm{.}
\eeqaa
Since $x(z)=(1-a(z))^{-1}$ for $z\in B_1(0)$, it follows from (\ref{Inv-20}) 
and (\ref{Inv-30}) that
\beqae
x(z)
=
\pbrack{\sum_{n=0}^{\infty}\alpha_nz^n}\pbrack{\sum_{n=0}^{\infty}\hat{q}_nz^n}
\textrm{.}
\eeqae
As a result of equating coefficients, we have that 
$\ds x_n=\sum_{k=0}^{n}\alpha_{n-k}\hat{q}_k$.
Therefore
\beqae
|x_n|\leq C \sum_{k=0}^{\infty}|\hat{q}_k|\textrm{,}
\eeqae
where $C$ is a constant which is an upper bound for the sequence 
$(|\alpha_n|)$. So $(x_n)$ is bounded and then, by theorem \ref{Teo-Higidio},
the null solution of (\ref{Eq-Recu-10}) is stable. Conversely, suppose that the
null solution of (\ref{Eq-Recu-10}) is stable and $m_1\geq 2$. Replacing $z$ by
$e^{i\theta_1}z$ in $x(z)(1-a(z))=1$, it follows that
\beqae
p(z)(1-z)^{m_1-1}q\left(e^{i\theta_1}z\right)=1\textrm{,}
\quad\forall z\in B_1(0)\textrm{,}
\eeqae
with
\beqae
p(z)
=
(1-z)\left(1-e^{i\left(\theta_1-\theta_2\right)}z\right)^{m_2}
\cdots 
\left(1-e^{i\left(\theta_1-\theta_s\right)}z\right)^{m_s}x\left(e^{i\theta_1}z\right)
\textrm{.}
\eeqae
Since, by theorem \ref{Teo-Higidio}, the sequence $(x_n)$ is bounded, by 
applying the item 1 several times and finally the item 2 of the preceding lemma,
one has that $p(z)$ is bounded on $[0,1[$. Therefore
\beqae
1
=
\lim_{\underset{z\in\R \ }{z\con 1^-}}x(e^{i\theta_1}z)(1-a(e^{i\theta_1}z))
=
\lim_{\underset{z\in\R \ }{z\con 1^-}}p(z)(1-z)^{m_1-1}q(e^{i\theta_1}z)=0
\textrm{,}
\eeqae
which is absurd. Hence $m_1=1$.
\end{proof}
\begin{corollary}
Let $\ds\sum_{n=1}^{\infty}n|a_n|<\infty$. If $1-a(z)$ is not zero in $B_1(0)$
and has a finite number of zeros of order one in $|z|=1$, then the null 
solution of (\ref{Eq-Recu-10}) is stable.
\end{corollary}
\begin{proof}
Suppose that $z=1$ is a zero of $1-a(z)$ and consider
\beqae
q(z)
=
\sum_{n=0}^{\infty}q_nz^n:=\frac{1-a(z)}{1-z}
=
(1-a(z))\pbrack{\sum_{n=0}^{\infty}z^n}
\textrm{.}
\eeqae
Then $q_0=1$ and, for $n\geq 1$, one has that
\beqae
q_n=1-\sum_{k=1}^{n}a_k=\sum_{k=n+1}^{\infty}a_k\textrm{.}
\eeqae
It is easy to verify that, for $n\geq 1$, we have
\beqae
\sum_{k=0}^{n-1}q_k=\sum_{k=1}^{n}ka_k+n\sum_{k=n+1}^{\infty}a_k\textrm{.}
\eeqae
So
\beqae
\sum_{k=0}^{n-1}|q_k|
\leq 
\sum_{k=1}^{n}k|a_k|+n\sum_{k=n+1}^{\infty}|a_k|
\leq 
\sum_{k=1}^{\infty}k|a_k|
\quad 
\forall n\in\N
\textrm{.}
\eeqae
Thus $1-a(z)=(1-z)q(z)$ with $q\in\mathbb{L}^1$. On the other hand, if 
$z=e^{i\theta}$ is a zero of $1-a(z)$, one has that $z=1$ is a zero of
$1-\tilde{a}(z)$, $\tilde{a}(z)=a(e^{i\theta}z)$, which satisfies the hypotheses
of this corollary. Hence $1-\tilde{a}(z)=(1-z)\tilde{q}(z)$ with 
$\tilde{q}\in\mathbb{L}^1$ or, equivalently, $(1-a(z))=(1-e^{-i\theta}z)q(z)$,
where $q(z)=\tilde{q}(e^{-i\theta}z)$ and therefore $q\in\mathbb{L}^1$. Now 
consider the set $\{e^{i\theta_j}\,:\,j=1,\ldots,s\}$ consisting of all zeros of
$1-a(z)$. Then, by partial fractions, we have
\beqae
\frac{1-a(z)}{(1-e^{-i\theta_1}z)\cdots(1-e^{-i\theta_s})}
=
\sum_{j=1}^{s}\frac{\beta_j(1-a(z))}{(1-e^{-i\theta_j}z)}
=
\sum_{j=1}^{s}\beta_jq_j(z)
\textrm{,}
\eeqae
where each $q_j \in\mathbb{L}^1$. So, if 
$\ds q=\sum_{j=1}^{s}\beta_jq_j\in\mathbb{L}^1$, then
\beqae
1-a(z)=(1-e^{-i\theta_1}z)\cdots(1-e^{-i\theta_s}z)q(z)\textrm{.}
\eeqae
It follows from the preceding theorem that the null solution of 
(\ref{Eq-Recu-10}) is stable.
\end{proof}
\example The sequence $\ds a_n=c_0\frac{(-1)^n}{n^3}$, 
$\ds c_0:=\pbrack{\sum_{n=1}^{\infty}\frac{1}{n^3}}^{-1}$, satisfies the 
hypotheses of the previous corollary. So the null solution of (\ref{Eq-Recu-10})
is stable. Note that, in this case, theorem $\textrm{2}'$
cannot be used for obtaining this result since $R=1$.
\section{Stability via Approximation}
\noindent In this final section we state some conditions for stability via 
polynomial approximation by applying the following theorem (known as Rouch\'e's
Theorem):
\begin{theorem}
If $f$ and $f+h$ are analytic functions on $\overline{B_\rho(z_0)}$ such that
\beqae
|h(z)|<|f(z)|\quad\text{in}\quad|z|=\rho\textrm{,}
\eeqae
then $f$ and $f+h$ have the same number of zeros in $B_\rho(z_0)$.
\end{theorem}
\noindent Now, for each $n\in\N$, consider the polynomial
\beqae
p_n(z)=z^n-a_1z^{n-1}-\cdots-a_{n-1}z-a_n\textrm{,}
\eeqae
where $a_1$, $\ldots$, $a_n$ are the first $n$ coefficients of the power series
$a(z)$. Define
\beqae
r_n:=\max\{|z|\,:\,p_n(z)=0\}
\eeqae
and  $z_1,\ldots,z_n$ as the $n$ zeros of $p_n(z)$, that is,
\beqae
p(z)=(z-z_1)\ldots(z-z_n)\textrm{.}
\eeqae
In what follows we enunciate some results of stability via finite 
approximations of the characteristic equation.
\begin{theorem} If there exists an index $n$ such that
\beqae
r_n<1\quad\textrm{and}\quad\sum_{i=n+1}^{\infty}|a_i|<(1-r_n)^n\textrm{,}
\eeqae
then the null solution of (\ref{Eq-Recu-10}) is asymptotically stable.
\end{theorem}
\begin{proof}
For $|z|=1$, we have that
\beqae
1-r_n\leq 1-|z_i|\leq|z|-|z_i|\leq|z-z_i|\textrm{,}\quad i=1,\ldots,n\textrm{.}
\eeqae
So $(1-r_n)^n\leq|p_n(z)|$ for $|z|=1$. Consider the $n$th partial sum of 
$1-a(z)$, that is,
\beqae
s_n(z):=1-\sum_{k=1}^{n}a_k z^k\textrm{.}
\eeqae
Hence, since $s_n(z)=z^np_n(1/z)$ for $z\ne0$, one has that $s_n(z)$ is not zero
in $\overline{B_1(0)}$ and $(1-r_n)^n\leq|s_n(z)|$ for $|z|=1$. Then
\beqae
\left|1-a(z)-s_n(z)\right|
=
\left|\sum_{i=n+1}^{\infty}a_iz^i\right|
\leq
\sum_{i=n+1}^{\infty}\left|a_i\right|
<
\left(1-r_n\right)^n
\leq
\left|s_n(z)\right|
\textrm{.}
\eeqae
By Rouch\'e's Theorem, $1-a(z)$ and $s_n(z)$ have the same number of zeros in
$\overline{B_1(0)}$. Therefore $1-a(z)$ is not zero in $\overline{B_1(0)}$. It
follows from theorem \ref{Teo-Assint-10} that the null solution of 
(\ref{Eq-Recu-10}) is asymptotically stable.
\end{proof}
\example Let $(\beta_n)$ be a sequence with $\beta_n\in \{-1,1\}$. The sequence
\beqae
(a_n)=\pbrack{\frac{3}{2},-\frac{9}{16},\frac{\beta_1}{20},\frac{\beta_2}{20^2},\frac{\beta_3}{20^3},\cdots}
\eeqae
does not satisfy the hypothesis of theorem \ref{Teo-Abs-10}. However, by 
considering the polynomial $p_2(z)$, we obtain $r_2=3/4$ and, since
\beqae
\sum_{i=3}^{\infty}|a_i|=\frac{1}{19}<(1-3/4)^2\textrm{,}
\eeqae
the null solution of (\ref{Eq-Recu-10}) is asymptotically stable.\\
\example The sequence
\beqae
(a_n)=\pbrack{1,-\frac{41}{36},\frac{8}{9},-\frac{34}{81},\frac{16}{81},-\frac{4}{81},\frac{1}{2\cdot 4^6},\frac{1}{2^2\cdot 4^6},\frac{1}{2^3\cdot 4^6},\cdots}
\eeqae
does not satisfy the hypothesis of theorem \ref{Teo-Abs-10}. By a computational
calculus, the values of $r_n$ and $\ds L_n:=\sum_{i={n+1}}^{\infty}|a_i|$ are as
shown in the table that follows. Note that the hypothesis of the preceding 
theorem is satisfied for $n=6$. So the null solution of (\ref{Eq-Recu-10}) is 
asymptotically stable.
\begin{center}
\renewcommand*{\arraystretch}{1.4}
\begin{tabular}{|c|c|c|c|}
\hline
    $n$ & $r_n$ & $L_n$  & $(1-r_n)^n$\\\hline
    $1$ & $1$ & - & - \\
    $2$ & $1.067$ & - & - \\
    $3$ & $1.012$ & - & - \\
    $4$ & $0.913$ & $0.24716$ & $0.00005$\\
    $5$ & $0.781$ & $0.04963$ & $0.00050$\\
    $6$ & $0.667$ & $0.00024$ & $0.00137$\\\hline
\end{tabular}
\end{center}
\begin{theorem}
If there exists an index $n$ such that $r_n>1$ and
\beqae
\sum_{i=n+1}^{\infty}|a_i|<\delta_n(\rho_0)\textrm{,}
\eeqae
where $\rho_0$ is the point that maximizes the function 
$\delta_n:[r_n^{-1},1]\con\R$ defined by
\beqae
\delta_n(\rho):=|(1-\rho|z_1|)(1-\rho|z_2|)\cdots(1-\rho|z_n|)|\textrm{,}
\eeqae
then the null solution of (\ref{Eq-Recu-10}) is unstable.
\end{theorem}
\begin{proof}
Since $\delta(r_n^{-1})=0$, one has $r_n^{-1}<\rho_0\leq 1$. If $|z|=\rho_0$,
then $|1-zz_i|\geq |1-\rho_0|z_i||$. So the partial sum considered in the 
preceding theorem satisfies
\beqae
|s_n(z)|=|1-zz_1|\cdots|1-zz_n|\geq \delta_n(\rho_0)\textrm{.}
\eeqae
Therefore, for $|z|=\rho_0$, we have that
\beqae
|1-a(z)-s_n(z)|=|\sum_{i=n+1}^{\infty}a_iz^i|\leq\sum_{i=n+1}^{\infty}|a_i|<\delta_n(\rho_0)\leq|s_n(z)|\textrm{.}
\eeqae
Now, let $j$ with $|z_j|=r_n$. Then $z_j^{-1}\in B_{\rho_0}(0)$ and
$s_n(z_j^{-1})=0$. Hence, by Rouch\'e's Theorem, $1-a(z)$ has at least a zero in
$B_{\rho_0}(0)\subset B_{1}(0)$. Therefore, by theorem \ref{Teo-NoEstavel-10}, 
the null solution of (\ref{Eq-Recu-10}) is  unstable.
\end{proof}
\noindent{\bf Remark:} Let us put the moduli of the zeros of $p_n$ in 
descending order, say $|z_1|\geq\cdots\geq|z_n|$. Assume that the hypotheses of 
the preceding theorem hold. Consider $i_0$ is the highest index with 
$|z_{i_0}|>1$. Then $|z_{i_0+1}|\leq1$. Hence one has the following estimates:
\begin{enumerate}
\item If $i_0=n$, then $E_1:=|1-|z_{n}||^n\leq\delta_n(1)\leq\delta_n(\rho_0)$.
\item If $|z_{i_0+1}|<1$, then
$
E_2
:=
\min\{|1-|z_{i_0}||^n,|1-|z_{i_0+1}||^n\}\leq\delta_n(1)\leq\delta_n(\rho_0)
$.
\item If $|z_{i_0+1}|=1$, then 
$E_3:=|1-\rho_1|z_{i_0}||^n\leq\delta_n(\rho_1)\leq\delta_n(\rho_0)$ with
$\rho_1=2/(|z_{i_0}|+1)$.
\end{enumerate}
As a consequence of the previous remark, we may state the following corollary:
\begin{corollary}
With the same assumptions of the preceding theorem, if
\beqae
\sum_{i=n+1}^{\infty}|a_i|<E\textrm{,}
\eeqae
where $E=E_1,E_2$ or $E_3$ are given as in the previous remark, then the null 
solution of (\ref{Eq-Recu-10}) is unstable.
\end{corollary}
\example Consider the sequence $(a_n)$ given by
\beqae
a_1=4\textrm{,}\quad a_2=-4\textrm{,}\quad a_n=\frac{1}{2^{n-1}}\textrm{,}\quad
n\geq 3\textrm{.}
\eeqae
$(a_n)$ does not satisfy the assumptions of corollary \ref{MaiorUm-10}. However
the zeros of $p_2(z)$ are $z_1=z_2=2$. Since
\beqae
\sum_{k=3}^{\infty}|a_k|=\frac{1}{2}<|1-2|^2\textrm{,}
\eeqae
the null solution of (\ref{Eq-Recu-10}) is unstable by the previous corollary.


\begin{thebibliography}{99}

\bibitem{Berenhaut}
K.S. Berenhaut and N.G. Vish, Equations of convolution type with monotone coefficients, Journal of Difference Equations and Applications 17 (2011) 555-566.

\bibitem{Choi}
S.K. Choi, Y.H. Goo, Y. Hoe and N.J. Koo, Asymptotic behavior of nonlinear Volterra difference systems, Bull. Korean Math. Soc. 44 (2007) 177-184.

\bibitem{Crisci-Vecchio}
M.A. Crisci, V.B. Kolmanovskii, E. Russo and A. Vecchio, Boundedness of Discrete Volterra Equations, Journal of Mathematical Analysis and Applications 211 (997) 106-130.

\bibitem{Elaydi-10}
S. Elaydi, Stability of Volterra difference equations of convolution type, Dinamical Systems, Nankai Ser. Pure Appl. Math. Theoret. Phys. 4 (1993) 66-72.

\bibitem{Elaydi-2005}
S. Elaydi, An Introduction to Difference Equations, Undergraduate Texts in Mathematics, Springer Verlag, 2005.

\bibitem{Elaydi-2009}
S. Elaydi, Stability and asymptoticity of Volterra difference equations. A progress report, J. Comp. Appl. Math. 228 (2009) 504-513.

\bibitem{Elaydi-MV}
S. Elaydi, E. Messina and A. Vecchio, A note on the asymptotic stability of linear Volterra difference equations of convolution type, J. Difference Equ. Appl. 13 (2007) 1079-1084. 

\bibitem{Erdos}
P. Erd{\"o}s, W. Feller and H. Pollard, A property of power series with positive coefficients, Bull. Amer. Math. Soc. 55 (1949) 201-204.

\bibitem{Kolmanovskii-K}
V.B. Kolmanovskii and N.P. Kosareva, Stability of Volterra Difference Equations, Differential Equations 37 (2001) 1773-1782.


\bibitem{Kolmanovskii}
V.B. Kolmanovskii, E. Castellanos Velasco and J.A. Torres Mu{\~n}oz, A survey: stability and boundedness of Volterra difference equations, Nonlinear Analysis 53 (2003) 861-928.

\bibitem{Vecchio}
E. Messina, Y. Muroya, E. Russo and A. Vecchio, Asymptotic Behavior of Solutions for Nonlinear Volterra Discrete Equations, Discrete Dynamics in Nature and Society (2008) Article ID 867623, 18 pp.


\bibitem{Nguyen}
N. Van Minh, On the asymptotic behaviour of Volterra difference equations, J. Difference Equ. Appl. 19 (2013) 1317-1330.

\bibitem{Nigmatulin}
R. Nigmatulin, Asymptotic behavior of solutions of a nonlinear Volterra difference equation, Int. Electron. J. Pure Appl. Math. 6 (2013) 123-125.


\bibitem{Tang}
X.H. Tang and Z. Jiang, Asymptotic behavior of Volterra difference equation, J. Difference Equ. Appl. 13 (2007) 25-40.

\end{thebibliography}
\end{document}